\documentclass{siamart171218}
\usepackage{lipsum}
\usepackage{amsfonts}
\usepackage{graphicx}
\usepackage{epstopdf}
\usepackage{algorithmic}
\ifpdf
  \DeclareGraphicsExtensions{.eps,.pdf,.png,.jpg}
\else
  \DeclareGraphicsExtensions{.eps}
\fi

\newsiamremark{remark}{Remark}
\newsiamremark{hypothesis}{Hypothesis}
\crefname{hypothesis}{Hypothesis}{Hypotheses}
\newsiamthm{claim}{Claim}

\renewcommand{\Im}{\operatorname{Im}}  
\newcommand{\Li}{\operatorname{Li}}    

\headers{Derivative Corrections to the Trapezoidal Rule}{Carl R. Brune}

\title{Derivative Corrections to the Trapezoidal Rule%
\thanks{Submitted to the editors August 14, 2018.
\funding{This work was supported in part by the U.S. Department of Energy,
Grants No.~DE-FG02-88ER40387 and No.~DE-NA0002905.}}}

\author{Carl R. Brune\thanks{Edwards Accelerator Laboratory,
  Department of Physics and Astronomy,
  Ohio University,
  Athens, Ohio 45701
  (\email{brune@ohio.edu}).} }

\usepackage{amsopn}

\ifpdf
\hypersetup{
  pdftitle={Derivative Corrections to the Trapezoidal Rule},
  pdfauthor={Carl R. Brune}
}
\fi

\begin{document}

\maketitle

\begin{abstract}
Extensions to the trapezoidal rule using derivative information
are studied for periodic integrands and integrals along the
entire real line. Integrands which are analytic within a half plane
or within a strip containing the path of integration are considered.
Derivative-free error bounds are obtained.
Alternative approaches to including derivative information are discussed.
\end{abstract}



\section{Introduction}

The trapezoidal rule for numerical quadrature is remarkably accurate
when applied to periodic integrands or integrals along the
entire real axis.
We consider here extensions of this method where derivative
information is taken into account.
Trefethen and Weideman~\cite{Tre14} have recently produced a thorough
review of the trapezoidal rule, but did not cover derivative information.
In this work, we generalize the term ``trapezoidal rule'' to mean any numerical
quadrature scheme that utilizes information about the function being
integrated at equally-spaced quadrature points, with the rule treating
every point on the same footing (e.g., with equal weight for a typical
linear rule).

The use of derivative information in numerical quadrature has been
reviewed by Davis and Rabinowitz~\cite[section~2.8]{Dav84}; for a more
recent example, see Burg~\cite{Bur12}.
The case for utilizing derivative information in quadrature becomes
compelling when derivatives at the quadrature points can be calculated
with significantly less effort than the alternative of evaluating
the integrand at additional quadrature points.
This may be the case, for example, if the integrand
satisfies a differential equation.
The derivative corrections may also be useful for error analysis
in high-precision numerical quadrature~\cite{Bai06}.
Davis and Rabinowitz also noted that the calculation of derivatives often
requires additional ``pencil work'' -- a complication that has now been
removed for the most part by the advent of computer algebra.
The application derivative information to the trapezoidal rule was
pioneered by Kress for periodic functions~\cite{Kre72a} and functions on
the real line that are analytic within a strip containing
the path of integration~\cite{Kre72b}.

This paper is organized as follows. We first consider in
section~\ref{sec:periodic} periodic functions
which are analytic either within a half plane or a strip, with examples
presented in sections~\ref{sec:example_periodic_complex}
and~\ref{sec:example_periodic_real}.
We then consider in section~\ref{sec:real_line} functions on the real line
which are analytic within a strip or half plane.
In section~\ref{sec:Dinf} we consider the limit in which a large
number of derivatives are included and finally in
section~\ref{sec:other_approaches} some other approaches to taking
derivative information into account are discussed.
To the best of our knowledge, the results for functions that are analytic
within a half plane and the material in sections~\ref{sec:Dinf}
and~\ref{sec:other_approaches} are new.
We have utilized the notation of Trefethen and Weideman~\cite{Tre14}
to the extent possible and the proofs given below draw significantly
from their paper.

\section{Integrals over a Periodic Interval}
\label{sec:periodic}

Let $v$ be a real or complex function with period $2\pi$ on the real line
and consider the definite integral
\begin{equation} \label{eq:I_trapezoidal}
I=\int_0^{2\pi} v(\theta) \, d\theta.
\end{equation}
The trapezoidal rule approximation for this integral is given
by~\cite[(3.2)]{Tre14}
\begin{equation} \label{eq:trap_periodic}
I_N = \frac{2\pi}{N} \sum_{j=1}^N \, v(\theta_j),
\end{equation}
where $N>0$ is the number of quadrature points and $\theta_j = 2\pi j/N$.

Assuming that $v$ is $D$-times differentiable, we define a generalized
trapezoidal rule approximation that takes into account derivative information
at the quadrature points via
\begin{equation} \label{eq:IND}
I_{N,D} = \frac{2\pi}{N} \sum_{j=1}^N \sum_{k=0}^D \left(\frac{1}{N}\right)^k
  A_{k,D} \, v^{(k)}(\theta_j),
\end{equation}
where $D$ is the maximum derivative order included
and $A_{k,D}$ are constants, with $A_{0,D}=1$.
Note that we have defined $A_{k,D}$ to be independent of the particular
point $j$, which is an intuitive choice based on the symmetry of the
points but not a requirement.
For simplicity, we have assumed that no derivatives are skipped in
the sum over $k$, but this also is not a requirement.
The factor of $(1/N)^k$ has been inserted for convenience: with this
factor, the prescriptions for defining $A_{k,D}$ given below lead to
$A_{k,D}$ being independent of $N$.
We observe that for $D=0$, the
standard trapezoidal rule for periodic functions (\ref{eq:trap_periodic}),
which does not use derivatives, is recovered.

\begin{theorem}
\label{thm:periodic_restricted}
Suppose $v$ is $2\pi$-periodic and analytic and satisfies $|v(\theta)|\le M$
in the half-plane $\Im\,\theta>-a$ for some $a>0$.
Further suppose that $D$ is a positive integer, $k$
is an integer with $0\le k\le D$, and
\begin{equation} \label{eq:Ak_theorem}
i^k A_{k,D} = \frac{(-1)^D}{D!} s(D+1,k+1) ,
\end{equation}
where $s(D+1,k+1)$ are the Stirling numbers of the first kind.
Then for $N>0$ and $I_{N,D}$ as defined in (\ref{eq:IND})
\begin{equation} \label{eq:bound_periodic_plane}
|I_{N,D}-I| \le \frac{2\pi M}{(e^{aN}-1)^{D+1}}
\end{equation}
and the constant $2\pi$ is as small as possible.
\end{theorem}

\begin{proof}
Since $v$ is analytic, it has the
uniformly and absolutely convergent Fourier series
\begin{equation} \label{eq:fourier}
v(\theta)=\sum_{\ell=-\infty}^{\infty} c_\ell e^{i\ell\theta},
\end{equation}
where the coefficients are given by
\begin{equation} \label{eq:fcoef_discrete}
c_\ell = \frac{1}{2\pi}\int_0^{2\pi} e^{-i\ell\theta}v(\theta)\,d\theta.
\end{equation}
From~(\ref{eq:I_trapezoidal}) and~(\ref{eq:fcoef_discrete}), we also have
\begin{equation} \label{eq:c0_discrete}
I=2\pi c_0.
\end{equation}
We define the auxiliary function
\begin{equation} \label{eq:vND}
v_{N,D}(\theta)=\sum_{k=0}^D \left(\frac{1}{N}\right)^k
  A_{k,D} \, v^{(k)}(\theta),
\end{equation}
which is also analytic and $2\pi$~periodic.
Using the expansion (\ref{eq:fourier}) we can then write
\begin{equation} \label{eq:vND_discrete}
v_{N,D}(\theta)= \sum_{\ell=-\infty}^{\infty} \sum_{k=0}^D
  A_{k,D} \left(\frac{i\ell}{N}\right)^k c_\ell e^{i\ell\theta},
\end{equation}
where we have used the fact that (\ref{eq:fourier}) is absolutely convergent
to justify differentiating and re-ordering the summation and
$\ell^k$ is understood to be unity if $\ell=k=0$.
Using the definition (\ref{eq:vND}), we may write (\ref{eq:IND}) as
\begin{equation}
I_{N,D} = \frac{2\pi}{N}  \sum_{j=1}^N v_{N,D}(\theta_j),
\end{equation}
which when combined with (\ref{eq:c0_discrete}) and
(\ref{eq:vND_discrete}) gives
\begin{equation} \label{eq:IND-I_first}
I_{N,D}-I = \frac{2\pi}{N} \sum_{j=1}^N \sum_{\ell=1}^\infty \sum_{k=0}^D
  A_{k,D} \left[ \left(\frac{i\ell}{N}\right)^k c_\ell e^{i\ell\theta_j} +
  \left(\frac{-i\ell}{N}\right)^k c_{-\ell} e^{-i\ell\theta_j} \right].
\end{equation}
Using the fact that
\begin{equation}
\sum_{j=1}^N e^{i\ell\theta_j} =
  \left\{\begin{array}{ll}
     N & \ell \bmod N = 0 \\
     0 & \text{otherwise} \end{array} \right.
\end{equation}
and redefining the index $\ell\rightarrow\ell N$, (\ref{eq:IND-I_first})
becomes
\begin{equation} \label{eq:IND-I}
I_{N,D}-I = 2\pi \sum_{\ell=1}^\infty \sum_{k=0}^D
  A_{k,D} \left[ (i\ell)^k c_{\ell N} +
  (-i\ell)^k c_{-\ell N} \right].
\end{equation}

The bound $|v(\theta)|\le M$ for $\Im\,\theta >-a$ provides a
constraint on the coefficients $c_\ell$, which may be quantified by
considering various integration contours for (\ref{eq:fcoef_discrete}).
For $\ell\ge 0$, shifting the interval [0,$2\pi$]
downward by a distance $a'<a$ into the lower half plane shows
$|c_\ell| \le Me^{-\ell a}$, where we have
taken $a'$ arbitrarily close to $a$ and noted that the contributions from
the sides of the contour vanish by periodicity.
For $\ell< 0$, the interval may be shifted upwards an arbitrary
distance $b$, which leads to $|c_\ell|\le Me^{\ell b}$.
Since $b$ is arbitrary, $c_\ell$ must vanish in this case.
Summarizing, we have
\begin{subequations}
\begin{alignat}{2} \label{eq:bound_cl_plane}
|c_\ell| & \le M e^{-\ell a} & \quad\quad & \ell\ge 0 \quad\quad \text{and} \\
c_\ell & = 0 && \ell < 0 .
\end{alignat}
\end{subequations}

With this restriction on the Fourier coefficients, (\ref{eq:IND-I}) now becomes
\begin{equation} \label{eq:IND-restricted}
I_{N,D}-I = 2\pi \sum_{\ell=1}^\infty \sum_{k=0}^D (i\ell)^k A_{k,D} \, c_{\ell N}.
\end{equation}
In view of the geometric decay of the Fourier coefficients,
we will choose the remaining $A_{k,D}$ to eliminate as many low-order
Fourier coefficients as possible from the right-hand-side of
(\ref{eq:IND-restricted}). We thus now require
\begin{equation} \label{eq:vander_restricted}
\sum_{k=1}^D (i\ell)^k A_{k,D} =-1 \quad\quad 1\le \ell \le D.
\end{equation}
This is an inhomogeneous Vandermonde system for $i^k A_{k,D}$,
which must have a unique non-trivial solution.
It is useful to consider the quantity
\begin{equation} \label{eq:E_ell}
E_{\ell,D} =  \sum_{k=0}^D (i\ell)^kA_{k,D} =\prod_{k=1}^D(1-\ell/k),
\end{equation}
where the factorization results from the definition $A_{0,D}=1$
and the observation that $E_{\ell,D}$ is a polynomial in $\ell$ of degree $D$
that according to the definition (\ref{eq:vander_restricted})
has zeros for the first $D$ positive integers.
We see that $i^DA_{D,D}=(-1)^D/D!$ and that
\begin{equation}
E_{\ell,D} = (1-\ell/D) E_{\ell,D-1} , \quad D\ge 1,
\end{equation}
which implies a recurrence formula:
\begin{equation} \label{eq:recur_Ak}
i^k A_{k,D} = i^k A_{k,D-1} -\frac{i^{k-1} A_{k-1,D-1}}{D},
  \quad 1\le k\le D-1, D\ge 1 .
\end{equation}
The coefficients may also be represented by
\begin{equation} \label{eq:Ak}
i^kA_{k,D} = \frac{(-1)^D}{D!} s(D+1,k+1),
\end{equation}
where $s(D+1,k+1)$ are the Stirling numbers of the first kind~\cite{Bre10}.
This result can be confirmed by noting that it correctly yields
$A_{0,D}=1$, $i^DA_{D,D}=(-1)^D/D!$, and, using the recurrence formula
for the Stirling numbers of the first kind~\cite[(26.8.18)]{Bre10}, satisfies
(\ref{eq:recur_Ak}).

Using the factorized form of $E_{\ell,D}$ we also find
\begin{equation} \label{eq:E_binomial}
E_{\ell,D} = (-1)^D\binom{\ell-1}{D} , \quad \ell>D
\end{equation}
and thus (\ref{eq:IND-restricted}) becomes
\begin{equation} \label{eq:IND-I-Eell}
I_{N,D}-I = 2\pi \sum_{\ell=D+1}^\infty c_{\ell N} E_{\ell,D} .
\end{equation}
Using the bound on the Fourier coefficients (\ref{eq:bound_cl_plane})
and (\ref{eq:E_binomial}), we then obtain
\begin{equation} \label{eq:bound_E_ell}
|I_{N,D}-I| \le 2\pi M \sum_{\ell=D+1}^\infty e^{-a\ell N}\binom{\ell-1}{D},
\end{equation}
which upon summing the series is (\ref{eq:bound_periodic_plane}).

To show the sharpness of the constant $2\pi$ in the bound
(\ref{eq:bound_periodic_plane}) we consider
\begin{equation}
v(\theta) = e^{i(D+1)N\theta},
\end{equation}
which has $I=0$ and has vanishing Fourier coefficients except for
$c_{(D+1)N}=1$ and leads to
\begin{equation}
I_{N,D}-I = 2\pi E_{D+1,D} = 2\pi\, (-1)^D .
\end{equation}
The sharp bound on this $v(\theta)$ for $\Im\,\theta>-a$ is
\begin{equation}
|v(\theta)| < M = e^{a(D+1)N} .
\end{equation}
The bound (\ref{eq:bound_periodic_plane}) is seen to be
asymptotic to the exact result for $|I_{N,D}-I|$ as $N\rightarrow\infty$.
\end{proof}

This result is an extension of Theorem~3.1 of Trefethen and
Weideman~\cite{Tre14}, which makes the same assumptions
regarding $v(\theta)$ and finds that the error of the usual
trapezoidal rule to be $|I_{N,0}-I|=O(e^{-aN})$ for $N\rightarrow\infty$.
When derivative information is included, we find that the rate of
geometric convergence can be improved to $O(e^{-a(D+1)N})$.
Practically speaking, one thus expects the number of quadrature points needed
to achieve a given level of precision to be reduced by a factor of $(D+1)$
when derivative information is considered.

We also observe that the bound (\ref{eq:bound_periodic_plane}) implies
\begin{equation}
\lim_{D\rightarrow\infty} |I_{N,D}-I|=0 \quad a,N>0,
\end{equation}
where the convergence is geometric.
However, there are practical issues when $D$ is large, as
there must be large cancellations in $I_{N,D}$ in this limit:
consider, for example,
\begin{equation} \label{eq:A_1}
iA_{1,D} = \sum_{k=1}^D \frac{1}{k}
\end{equation}
which diverges as $D\rightarrow\infty$.

In Table~\ref{tab:a_coeff} we present $A_{k,D}$ for
$D=1,$~2, and~3.
A numerical example of this quadrature formula is
provided below in section~\ref{sec:example_periodic_complex}.

\begin{table}[tp]
{\footnotesize
  \caption{The constants $A_{k,D}$ defined by (\ref{eq:Ak_theorem}),
  for the three lowest $D$ values. }
\label{tab:a_coeff}
\begin{center}
\begin{tabular}{c|ccc}
D & $A_{1,D}$ & $A_{2,D}$ & $A_{3,D}$
\rule[-0.5em]{0pt}{0.5em} \\ \hline
\rule{0pt}{1.0em}%
1 & $i$        & -          & -      \\
2 & $3i/2$     & $-1/2$     & -      \\
3 & $11i/6$    & $-1$       & $-i/6$ \\ \hline
\end{tabular}
\end{center}
}
\end{table}

Due to the restrictions on $v(\theta)$, this theorem is not applicable
to real integrands, unless they are a constant.
We will next consider a similar extension to Theorem~3.2 of Trefethen and
Weideman~\cite{Tre14}, which has a less restrictive condition on $v(\theta)$
and may be applied to real integrands.

\begin{theorem}
\label{thm:periodic_strip}
Suppose $v$ is $2\pi$-periodic and analytic and satisfies $|v(\theta)|\le M$
in the strip $|\Im\,\theta|<a$ for some $a>0$.
Further suppose that $D$ is a positive even integer, $\ell$ and
$m$ are integers with $1\le \ell,m\le D/2$,
$B_{0,D}=1$, $B_{2m-1,D}=0$, and
$B_{2m,D}$ are the solution to the Vandermonde system
\begin{equation} \label{eq:B2m}
\sum_{m=1}^{D/2} (-1)^{m}  \ell^{2m} B_{2m,D} = -1 .
\end{equation}
Then with $I_{N,D}$ as defined in (\ref{eq:IND})
with $B_{k,D}$ replacing $A_{k,D}$ therein and $N>0$
\begin{subequations} \label{eq:bound_period_strip_both}
\begin{equation} \label{eq:bound_periodic_strip}
|I_{N,D}-I| \le \frac{4\pi M}{(1-e^{-a N})^{D+1}} \left| \sum_{\ell=D/2+1}^{D+1}
  (-1)^\ell \binom{D+1}{\ell} e^{-a\ell N} \right| ,
\end{equation}
and for $N\rightarrow\infty$
\begin{equation} \label{eq:bound_periodic_strip_asymp}
|I_{N,D}-I| \le 4\pi M \binom{D+1}{D/2} e^{-a(D/2+1)N}
  \left[ 1 + O(e^{-aN}) \right] ,
\end{equation}
\end{subequations}
and the constant $4\pi$ is as small as possible.
\end{theorem}

\begin{proof}
The proof is very similar to Theorem \ref{thm:periodic_restricted}.
Equations (\ref{eq:fourier})-(\ref{eq:IND-I}) continuing to hold, with
$B_{k,D}$ replacing $A_{k,D}$.
The bound $|v(\theta)|\le M$ for $|\Im\, \theta|<a$ provides a weaker
constraint on the Fourier coefficients.
For $\ell\ge 0$, the bound on $c_\ell$ is unchanged.
For $\ell\le 0$, the integration interval in (\ref{eq:fcoef_discrete})
may be shifted upward by a distance $a'<a$ which leads to
$|c_\ell| \le Me^{\ell a}$. Summarizing, we now have
\begin{equation} \label{eq:bound_cl_strip}
|c_\ell| \le M e^{-|\ell|a}.
\end{equation}

In this case, the remainder (\ref{eq:IND-I}) now becomes
\begin{equation} \label{eq:IND-restricted-strip}
I_{N,D}-I = 2\pi \sum_{\ell=1}^\infty \sum_{k=0}^D
  B_{k,D} \left[ (i\ell)^k c_{\ell N} +
  (-i\ell)^k c_{-\ell N} \right].
\end{equation}
For a given value of $\ell$ in (\ref{eq:IND-restricted-strip}),
the Fourier coefficients appear in pairs, $c_{\ell N}$ and $c_{-\ell N}$,
that are of comparable magnitude.
We will again choose the remaining $B_{k,D}$ to eliminate as many of the
low-order Fourier components as possible.
In order make the contribution of a particular pair vanish, we require
\begin{subequations} \label{eq:vander-restricted-strip}
\begin{align}
& 1+\sum_{k=1}^D (i\ell)^k B_{k,D} = 0 \quad\quad \text{and} \\
& 1+\sum_{k=1}^D (-i\ell)^k B_{k,D} = 0 .
\end{align}
\end{subequations}
Adding or subtracting these equations decouples the even and odd coefficients:
\begin{subequations}
\begin{align} \label{eq:vander_even}
& \sum_{m=1}^{D/2} (-1)^{m}\ell^{2m} B_{2m,D} = -1 \quad\quad \text{and} \\
& \sum_{m=1}^{D/2} (-1)^{m} \ell^{2m-1} B_{2m-1,D} = 0 , \label{eq:vander_odd}
\end{align}
\end{subequations}
where were have now restricted $D$ to be even.
Because there are two equations for each $\ell$ value, this assumption
allows us to match the number of equations to the number of unknown $B_{k,D}$
by considering $\ell$ values from one up to $D/2$.
Since (\ref{eq:vander_even}) is an inhomogeneous real Vandermonde system for
$(-1)^{m}B_{2m,D}$, it has a unique non-trivial solution.
For the odd coefficients, (\ref{eq:vander_odd}) is a homogeneous Vandermonde
system for $(-1)^{m} B_{2m-1,D}$ and  its only solution is the
trivial one,
\begin{equation}
B_{2m-1,D} = 0 , \quad\quad 1\le m\le D/2.
\end{equation}
We note that if $D$ is permitted to be odd there is ambiguity in the definition
of $B_{2m,D}$ because considering $\ell$ values up to $(D-1)/2$ does not
provide enough equations to uniquely determine the coefficients, but
increasing the maximum $\ell$ value by one overdetermines them.

It is useful to consider the quantity
\begin{equation} \label{eq:F_ell}
F_{\ell,D} = \sum_{m=0}^{D/2} (-1)^m \ell^{2m} B_{2m,D}
 = \prod_{m=1}^{D/2}\left[1-(\ell/m)^2\right],
\end{equation}
where the product form results from noting that $F_{\ell,D}$ is a polynomial
in $\ell^2$ of degree $D/2$ with $F_{0,D}=1$ and that the fact that,
according to (\ref{eq:vander_even}), $F_{\ell,D}$ is zero for
when $\ell$ is one of the first $D/2$ positive integers.
In addition, we note that $F_{\ell,D}$ is nonzero and monotonically increasing
in absolute value for $\ell > D/2$.
From this equation, one can observe at once that
\begin{align} \label{eq:sum_inverse_squared}
B_{2,D} &= \sum_{m=1}^{D/2} \frac{1}{m^2} \quad\quad \text{and} \\
B_{D/2,D} &= \frac{1}{[(D/2)!]^2}.
\end{align}
Following Kress~\cite{Kre72a,Kre72b}, a recurrence relation for
the $B_{2m,D}$ coefficients may be derived by noting
\begin{equation} \label{eq:F_ell2}
F_{\ell,D} = \left[1-\left(2\ell/D\right)^2\right] F_{\ell,D-2}, \quad D\ge 2
\end{equation}
which provides
\begin{equation}
B_{2m,D} = B_{2m,D-2} + (2/D)^2 B_{2m-2,D-2}, \quad 1\le m\le D/2,
  D\ge 2.
\end{equation}
Using the factorized form of $F_{\ell,D}$, one readily finds
\begin{equation} \label{eq:F_ell_binomial}
F_{\ell,D} = (-1)^{D/2} \binom{\ell+D/2}{D/2}\binom{\ell-1}{D/2},
  \quad \ell > D/2
\end{equation}

With this definition for $F_{\ell,D}$, (\ref{eq:IND-restricted-strip})
becomes
\begin{equation} \label{eq:IND-I-Fell}
I_{N,D}-I = 2\pi \sum_{\ell=D/2+1}^\infty ( c_{\ell N} + c_{-\ell N} ) F_{\ell,D} .
\end{equation}
Using the bound on the Fourier coefficients (\ref{eq:bound_cl_strip}),
we then obtain
\begin{equation} \label{eq:bound_F_ell}
I_{N,D}-I = 4\pi M \left|\sum_{\ell=D/2+1}^\infty e^{-a\ell N} F_{\ell,D}
  \right| ,
\end{equation}
where we have used the fact that all $F_{\ell,D}$ have the same sign
for $\ell>D/2$ to justify moving the absolute value outside of the summation.
Making use of the identity
\begin{equation} \label{eq:sum_F_ell}
\begin{split}
(-1)^{D/2} \sum_{\ell=D/2+1}^\infty & \binom{\ell+D/2}{D/2}\binom{\ell-1}{D/2}
  e^{-a\ell N} \\ = & \frac{-1}{(1-e^{-a N})^{D+1}} \sum_{\ell=D/2+1}^{D+1}
  (-1)^\ell \binom{D+1}{\ell} e^{-a\ell N} ,
\end{split}
\end{equation}
one obtains (\ref{eq:bound_periodic_strip}), which is asymptotically
equivalent to the bound (\ref{eq:bound_periodic_strip_asymp})
as $N\rightarrow\infty$.

To show the sharpness of the constant $4\pi$ in the bounds
(\ref{eq:bound_periodic_strip}) and (\ref{eq:bound_periodic_strip_asymp}),
we consider
\begin{equation}
v(\theta)=2\cos(D/2+1)N\theta,
\end{equation}
which has $I=0$ and vanishing Fourier
coefficients except for $c_{\pm(D/2+1)N}=1$ and leads to
\begin{equation}
I_{N,D} - I = 4\pi \sum_{m=0}^{D/2} (-1)^m (D/2+1)^{2m} B_{2m,D} =
  4\pi (-1)^{D/2}\binom{D+1}{D/2}.
\end{equation}
The sharp bound on this $v(\theta)$ for $|\Im\,\theta|<a$ is
\begin{equation}
|v(\theta)| < M=2\cosh(D/2+1)Na .
\end{equation}
The bounds (\ref{eq:bound_periodic_strip}) and
(\ref{eq:bound_periodic_strip_asymp}) are both seen to be asymptotic to
the exact result for $|I_{N,D}-I|$ as $N\rightarrow\infty$.
\end{proof}

Theorem~3.2 of Trefethen and Weideman~\cite{Tre14}, which makes the same
assumptions regarding $v(\theta)$, finds that the error of the usual
trapezoidal rule to be $|I_{N,0}-I|=O(e^{-aN})$ for $N\rightarrow\infty$.
When derivative information is included, we find that the rate of
geometric convergence can be improved to $|I_{N,D}-I|=O(e^{-a(D/2+1)N})$.
Interestingly, the coefficients of the odd derivatives in (\ref{eq:IND})
are found vanish  -- which implies they are not useful
for improving the accuracy of the trapezoidal rule in this case.
This quadrature rule appears to have been  first derived by
Kress~\cite{Kre72a}.
Our error bound is somewhat tighter, as Kress (in our notation) utilized
\begin{equation} \label{eq:kress_approx}
\left| \sum_{\ell=D/2+1}^{D+1} (-1)^\ell \binom{D+1}{\ell} e^{-a\ell N} \right|
  \le 2^D e^{-a(D/2+1)N} ,
\end{equation}
which is only sharp for $D=0$.
The leading behavior of the error bound (\ref{eq:bound_periodic_strip_asymp})
is consistent with the findings of Wilhelmsen~\cite{Wil78}.
A numerical demonstration of this quadrature rule is provided below in
section~\ref{sec:example_periodic_real}.

\begin{table}[tp]
{\footnotesize
  \caption{The constants $B_{2m,D}$ defined by (\ref{eq:B2m})
  and $F_{D/2+1,D}$ defined by (\ref{eq:F_ell_binomial}) which are
  applicable to Theorems~\ref{thm:periodic_strip} and~\ref{thm:infinite_strip},
  for the three lowest even $D$ values. }
\label{tab:b_coeff}
\begin{center}
\begin{tabular}{c|ccc|c}
D & $B_{2,D}$ & $B_{4,D}$ & $B_{6,D}$ &
  $F_{D/2+1,D} = (-1)^{D/2}\binom{D+1}{D/2}$ \\ \hline
\rule{0pt}{1.3em}%
2 & 1          & -          & -          & -3 \\
4 & 5/4        & 1/4        & -          & 10 \\
6 & 49/36      & 7/18       & 1/36       & -35 \\ \hline
\end{tabular}
\end{center}
}
\end{table}

The  polylogarithm function
\begin{equation} \label{eq:polylog}
\Li_{-k}(z) = \sum_{\ell=1}^\infty \ell^k z^\ell , \quad |z|<1
\end{equation}
may be used to write remainder bound in third form, in addition to
(\ref{eq:bound_periodic_strip}) or (\ref{eq:bound_F_ell}) with
(\ref{eq:F_ell_binomial}) for $F_{\ell,D}$.
Using (\ref{eq:F_ell}) for $F_{\ell,D}$ in (\ref{eq:bound_F_ell})
with (\ref{eq:polylog}), we have
\begin{equation}
\sum_{\ell=D/2+1}^\infty e^{-a\ell N} F_{\ell,D} = 
\sum_{m=0}^{D/2} (-1)^m B_{2m,D} \Li_{-2m}(e^{-aN}). 
\end{equation}

For the case $D=2$, we have $B_{2,2}$=1 and (\ref{eq:bound_periodic_strip})
becomes
\begin{equation}
|I_{N,2}-I| \le 4\pi M \frac{e^{-2aN}(3-e^{-aN})}{(1-e^{-aN})^3}.
\end{equation}
In Table~\ref{tab:b_coeff} we present $B_{2m,D}$ and $F_{D/2+1,D}$ for
$D=2,$~4, and~6.

\section{Example: Integral of a Periodic Complex Function}
\label{sec:example_periodic_complex}

\begin{figure}[htbp]
  \centering
  \includegraphics[width=0.7\columnwidth]{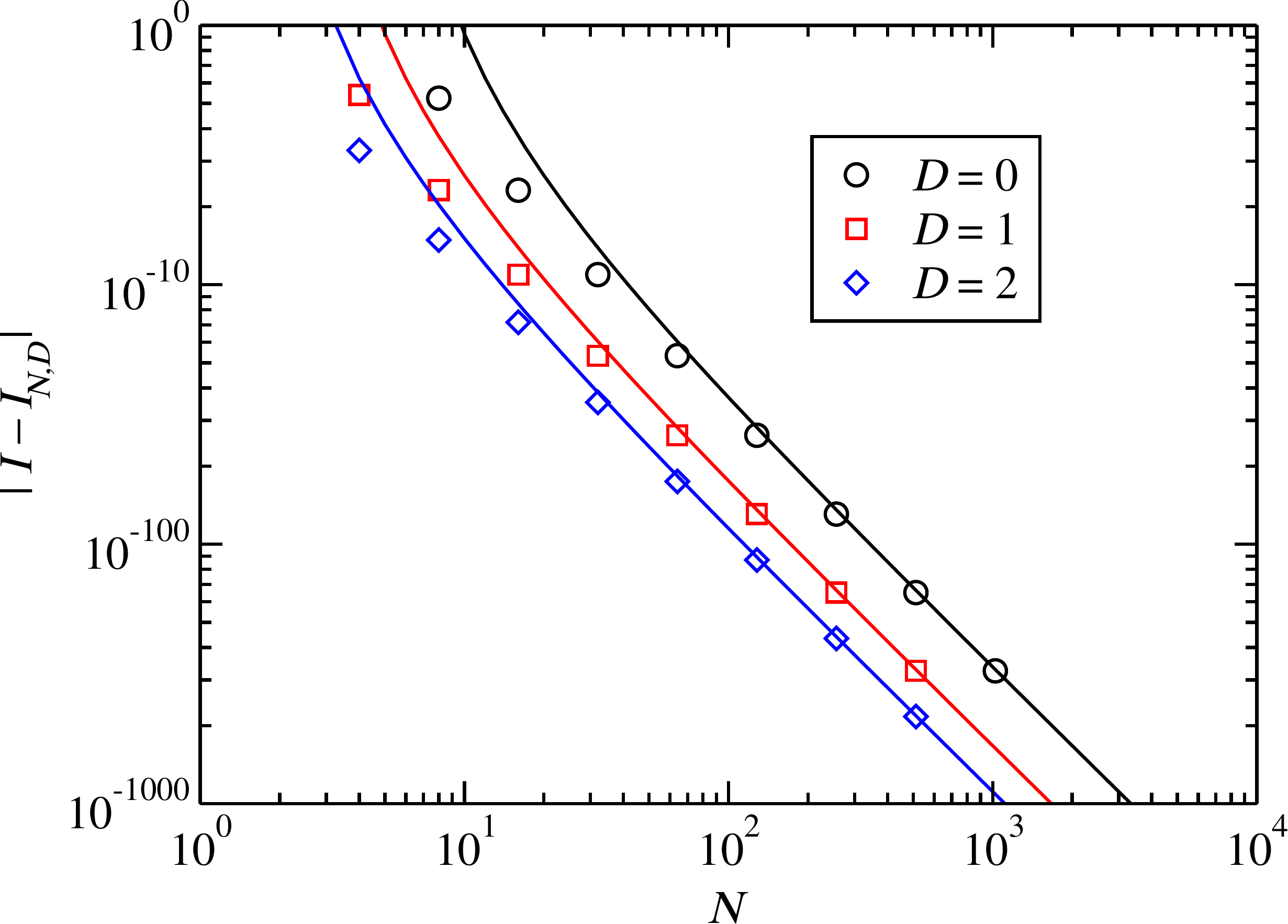}
  \caption{The actual (points) and predicted (curves) convergence
    of the generalized trapezoidal rule (\ref{eq:IND}) using
    the coefficients corresponding to Theorem~\ref{thm:periodic_restricted}
    given in Table~\ref{tab:a_coeff},
    for $v(\theta)$ given by
    (\ref{eq:v_ex_periodic_complex}) with $e^b=2$.}
  \label{fig:ex:periodic_complex}
\end{figure}

Here we present an example using a complex periodic function that
fulfills the requirements of Theorem~\ref{thm:periodic_restricted}:
\begin{equation} \label{eq:v_ex_periodic_complex}
v(\theta) = \frac{1}{e^b+e^{i\theta}},
\end{equation}
where $b$ is a positive real constant. This function has simple poles in the
lower half plane at $\theta=2\pi(j+1/2)-ib$, where $j$ is any integer.
We then have $0<a<b$, where $a$ defines the half plane in the conditions
of Theorem~\ref{thm:periodic_restricted}. For $\Im\,\theta>-a$, the
sharp upper bound on $|v(\theta)|$ is $M=(e^b-e^a)^{-1}$.
The error bound  may be optimized by choosing $a$ to minimize the leading
geometric term in (\ref{eq:bound_periodic_plane}), $2\pi Me^{-a(D+1)N}$.
Using calculus, one thus obtains
\begin{equation} \nonumber
a=b-\frac{1}{(D+1)N}+O(1/N^2), \quad N\rightarrow\infty \quad\text{and}
\end{equation}
\begin{equation}
|I_{N,D}-I| \le 2\pi e(D+1)N e^{-b[(D+1)N+1]} \left[1+O(1/N)\right],
  \quad N\rightarrow\infty .
\end{equation}
The actual convergence results and this bound are plotted in
Figure~\ref{fig:ex:periodic_complex}, for $e^b=2$.
The expected geometric convergence and improvement from including derivative
information are seen.
For this $v(\theta)$, the exact error can be calculated via
(\ref{eq:IND-I-Eell}), which results in
\begin{equation}
|I_{N,D}-I| = 2\pi e^{-b[(D+1)N+1]} \left[1+O(1/N)\right],
  \quad N\rightarrow\infty .
\end{equation}
We see that for large $N$, the error bound is a factor of $(D+1)Ne$
greater than the actual error.

\section{Example: Integral of a Periodic Real Function}
\label{sec:example_periodic_real}

\begin{figure}[htbp]
  \centering
  \includegraphics[width=0.7\columnwidth]{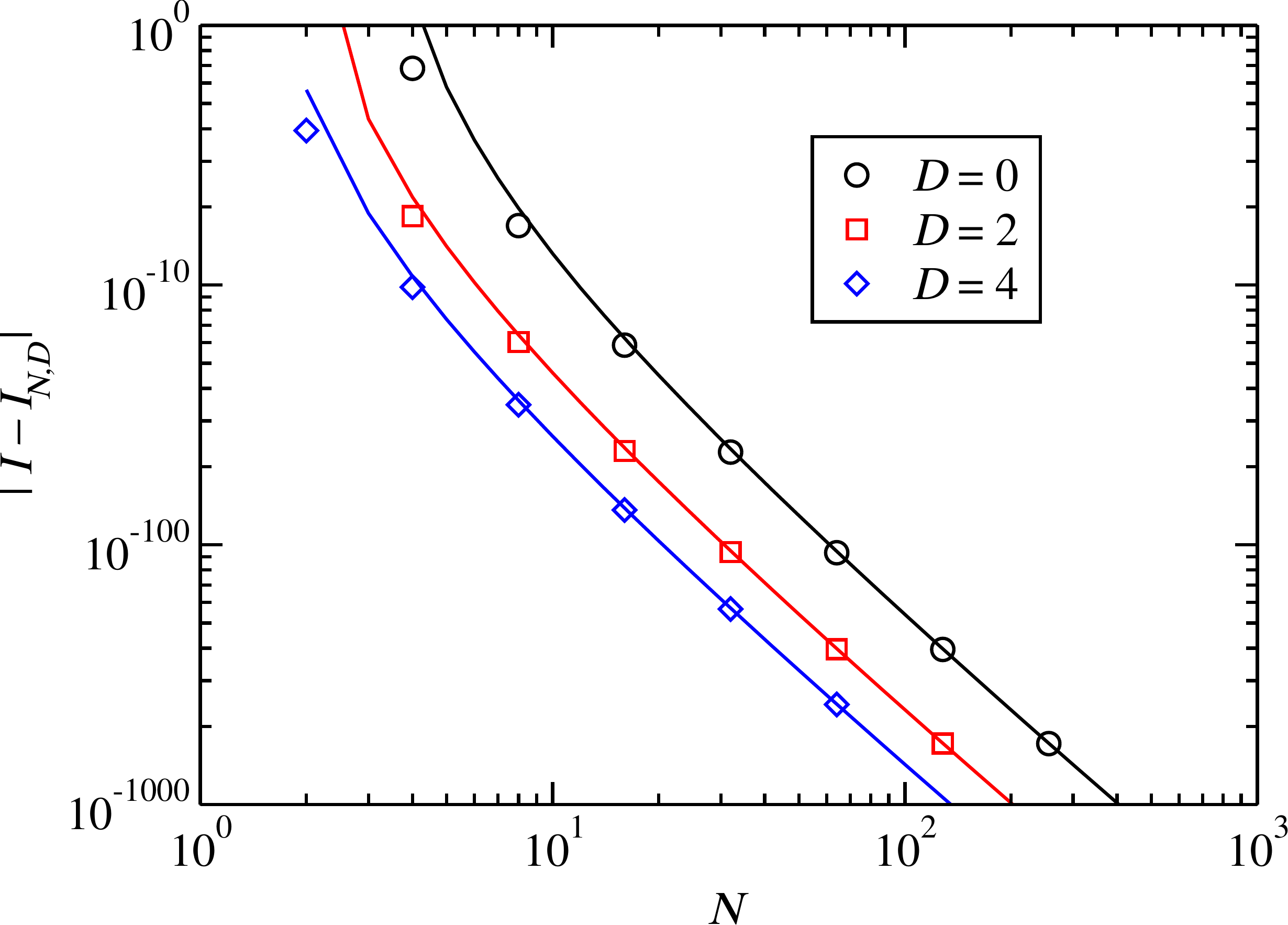}
  \caption{The actual (points) and predicted (curves) convergence
    of the generalized trapezoidal rule (\ref{eq:IND}) using
    the coefficients corresponding to Theorem~\ref{thm:periodic_strip}
    given in Table~\ref{tab:a_coeff},
    for $v(\theta)$ given by
    (\ref{eq:v_ex_periodic_real}).}
  \label{fig:ex:periodic_real}
\end{figure}

Here we present an example with a real integrand that
fulfills the requirements of Theorem~\ref{thm:periodic_strip}:
\begin{equation} \label{eq:v_ex_periodic_real}
v(\theta) = e^{\cos\theta},
\end{equation}
an example also considered by Trefethen and Weideman~\cite{Tre14}.
We first note the remarkable accuracy that can be achieved with just
a modest number of terms -- for example, $N=4$ and $D=4$ results in
\begin{equation}
I_{4,4} = \frac{\pi}{1024}\left(1101+553/e+474e\right) =
  7.9549265210781\ldots
\end{equation}
where the first 11 digits are correct.
In this case $v(\theta)$ is entire, with $|v(\theta)|$ unbounded
as $\Im\,\theta\rightarrow\infty$. 
The sharp bound on $|v(\theta)|$ in the strip $|\Im\,\theta|<a$
is $M=e^{\cosh a}$. The leading geometric term in the error bound
is $4\pi\binom{D+1}{D/2}Me^{-a(D/2+1)N}$, which may be minimized
using calculus, resulting in
\begin{equation} \nonumber
e^a=(D+2)N+O(1/N), \quad N\rightarrow\infty \quad\text{and}
\end{equation}
\begin{equation}
|I_{N,D}-I| \le 4\pi \binom{D+1}{D/2}
  \left[\frac{e}{(D+2)N}\right]^{(D/2+1)N} \left[1+O(1/N)\right],
  \quad N\rightarrow\infty .
\end{equation}
The actual convergence results and this bound are plotted in
Figure~\ref{fig:ex:periodic_real}, where
the expected geometric convergence and improvement from including derivative
information are seen.

\section{Integrals on the Real Line}
\label{sec:real_line}

Let $w$ be a real or complex function on the real line and consider the
definite integral
\begin{equation} \label{I_real_line}
I=\int_{-\infty}^\infty w(x) \, dx .
\end{equation}
The trapezoidal rule approximation for this integral is given by
\cite[(5.2)]{Tre14}
\begin{equation} \label{eq:trap_infinite}
I_h = h \sum_{j=-\infty}^{\infty} w(x_j) ,
\end{equation}
where $h>0$ and $x_j=jh$.

Assuming that $w$ is $D$-times differentiable, we define a generalized
trapezoidal rule approximation that takes into account derivative
information via
\begin{equation} \label{eq:IhD}
I_{h,D} = h \sum_{j=-\infty}^\infty \sum_{k=0}^D \left(\frac{h}{2\pi}\right)^k
  B_{k,D} \, w^{(k)}(x_j),
\end{equation}
where $D$ is the maximum derivative order included and
$B_{k,D}$ are constants with $B_{0,D}=1$.
We have assumed that $B_{k,D}$ is independent of $j$ and that no derivatives
are skipped in the sum over $k$, neither of which is a requirement.
The factor of $(h/2\pi)^k$ has been inserted for convenience, as it will
lead to $B_{k,D}$ being independent of $h$.
We observe that for $D=0$, the standard
trapezoidal rule (\ref{eq:trap_infinite}) is recovered.

\begin{theorem}
\label{thm:infinite_strip}
Suppose $w$ is analytic in the strip $|\Im \, x|<a$ for some $a>0$,
$w(x)\rightarrow 0$ uniformly as $|x|\rightarrow\infty$,
and for some $M$, it satisfies
\begin{equation}
\int_{-\infty}^\infty |w(x+ib)|dx \le M
\end{equation}
for all $b\in(-a,a)$. Further suppose that $D$ is a positive even integer,
$\ell$ and $m$ are integers with $1\le \ell,m\le D/2$,
$B_{0,D}=1$, $B_{2m-1,D}=0$, and
$B_{2m,D}$ are the solution to the Vandermonde system
\begin{equation} \label{eq:B2m_infinite}
\sum_{m=1}^{D/2} (-1)^{m}  \ell^{2m} B_{2m,D} = -1 .
\end{equation}
Then for $h>0$, $I_{h,D}$ as defined in (\ref{eq:IhD}) exists and 
\begin{subequations} \label{eq:boud_infinite_both}
\begin{equation} \label{eq:bound_infinite}
|I_{h,D}-I| \le \frac{2M}{(1-e^{-2\pi a/h})^{D+1}} \left| \sum_{\ell=D/2+1}^{D+1}
  (-1)^\ell \binom{D+1}{\ell} e^{-2\pi\ell a/h} \right| ,
\end{equation}
and for $h\rightarrow 0$
\begin{equation} \label{eq:bound_infinite_asymp}
|I_{h,D}-I| \le 2M \binom{D+1}{D/2}  e^{-2\pi (D/2+1) a/h}
  \left[ 1 + O(e^{-2\pi a/h}) \right] ,
\end{equation}
\end{subequations}
and the constant $2M$ is as small as possible.
\end{theorem}

\begin{proof}
The proof is by residue calculus. We definite the auxiliary function
\begin{equation}
w_{h,D}(x) = \sum_{k=0}^D \left(\frac{h}{2\pi}\right)^k B_{k,D} \, w^{(k)}(x).
\end{equation}
The assumption that $w(x)\rightarrow 0$ uniformly as $|x|\rightarrow\infty$
implies by Cauchy integrals that the same holds true for
$w^{(k)}(x)$ and $w_{h,D}(x)$.
The function
\begin{equation}
m(x) = -\frac{i}{2}\cot\frac{\pi x}{h}
\end{equation}
has simple poles at $x=0, \pm h, \pm 2h,\ldots,$ all with residues equal
to $h/(2\pi i)$.
For convenience, we consider the sum in (\ref{eq:IhD}) to be
symmetric, from $-n$ to $n$ with $n\rightarrow\infty$.
Our arguments are trivially generalized to an arbitrary
sum from $n_-$ to $n_+$, with $n_-,n_+\rightarrow\infty$, as our reasoning
do not depend upon the symmetry of the sum.

The residue theorem thus implies that for any positive integer $n$
\begin{equation}
I_{h,D}^{[n]} = \int_\Gamma m(x)w(x) \, dx,
\end{equation}
where $I_{h,D}^{[n]}$ is the truncated form of the generalized trapezoidal
rule (\ref{eq:IhD}) and the clockwise contour $\Gamma$ encircles the poles
in $[-nh,nh]$. We take $\Gamma$ to be the rectangular contour with vertices
$\pm(n+\frac{1}{2})h+ia'$ and $\pm(n+\frac{1}{2})-ia'$ for any $a'$ with
$0<a'<a$. This contour is depicted in Figure~5.1 of
Trefethen and Weideman~\cite{Tre14}.
We can also write using Cauchy's theorem that
\begin{equation} \label{eq:two_forms}
\int_{-(n+\frac{1}{2})h}^{(n+\frac{1}{2})h} w_{h,D}(x)\,dx = \int_{\Gamma_-} w_{h,D}(x)\,dx
  = -\int_{\Gamma_+} w_{h,D}(x) \, dx ,
\end{equation}
where $\Gamma_-$ and $\Gamma_+$ are the segments of $\Gamma$ with
$\Im\, x\le 0$ and $\Im\, x\ge 0$, respectively.
Using the average of these two forms of (\ref{eq:two_forms}) we can write
\begin{align} \label{eq:IhD-I-n}
h &\sum_{j=-n}^n w_{h,D}(jh) -\int_{-(n+\frac{1}{2})h}^{(n+\frac{1}{2})h} w_{h,D}(x)\,dx \\
  & = -\frac{1}{2}\int_{\Gamma_-} (1+i\cot\frac{\pi x}{h} )w_{h,D}(x) \, dx
      +\frac{1}{2}\int_{\Gamma_+} (1-i\cot\frac{\pi x}{h} )w_{h,D}(x) \, dx
  \nonumber \\
  & = -\int_{\Gamma_-} \frac{w_{h,D}(x)}{1-e^{2\pi ix/h}} \, dx
      +\int_{\Gamma_+} \frac{w_{h,D}(x)}{1-e^{-2\pi ix/h}} \, dx . \nonumber
\end{align}
In the limit $n\rightarrow\infty$, the contributions of the vertical legs of
the contours $\Gamma_\pm$ vanish. This can be seen by considering
$|1+\exp(\mp 2\pi ix/h)| \ge 2$ on the vertical legs of $\Gamma_\pm$
and the decay properties of $w_{h,D}(x)$.
We also have
\begin{align}
\int_{-\infty}^\infty w_{h,D}(x)\,dx & = I+\sum_{k=1}^D
  \left(\frac{h}{2\pi}\right)^k B_{k,D} \int_{-\infty}^\infty w^{(k)}(x) \, dx \\
  &= I, \nonumber
\end{align}
since for $k\ge 1$ the integrals on the right-hand side can be evaluated via
integration by parts and the results vanish do to the decay properties
of $w^{(k-1)}(x)$.
In the limit $n\rightarrow\infty$, (\ref{eq:IhD-I-n}) thus becomes
\begin{equation} \label{eq:IhD-I}
I_{h,D}-I = -\int_{-\infty-ia'}^{\infty-ia'} \frac{w_{h,D}(x)}{1-e^{2\pi i x/h}} \, dx
           -\int_{-\infty+ia'}^{\infty+ia'} \frac{w_{h,D}(x)}{1-e^{-2\pi i x/h}} \, dx .
\end{equation}

We define
\begin{equation}
f_{\pm}(x) = \frac{-1}{1-e^{\mp2\pi ix/h}} = \sum_{\ell=1}^\infty
  e^{\pm 2\pi i \ell x/h},
\end{equation}
where the geometric series representations are absolutely convergent
along the respective paths of integration in (\ref{eq:IhD-I}).
We also have
\begin{equation}
\int_{-\infty\pm ia'}^{\infty\pm ia'} w^{(k)}(x)\, f_{\pm}(x) \, dx =
  (-1)^k \int_{-\infty\pm ia'}^{\infty\pm ia'} w(x)\, f_{\pm}^{(k)}(x) \, dx ,
\end{equation}
using integration by parts. The surface terms vanish due to the decay
properties of $w^{(k)}(x)$ and the fact that $f_{\pm}(x)$ and its derivatives
are bounded as $x\rightarrow \pm\infty-ia'$ and $x\rightarrow \pm\infty+ia'$.
We can now write
\begin{align}
I_{h,D}-I &= \int_{-\infty-ia'}^{\infty-ia'} \sum_{\ell=1}^\infty \sum_{k=0}^D
  B_{k,D} ( i \ell)^k e^{-2\pi i \ell x/h} w(x) \, dx \\ \nonumber
&+ \int_{-\infty+ia'}^{\infty+ia'} \sum_{\ell=1}^\infty \sum_{k=0}^D
  B_{k,D} (-i \ell)^k e^{2\pi i \ell x/h} w(x) \, dx.
\end{align}
The bound on $w(x)$ implies that
\begin{equation} \label{eq:bound_int_we}
\left| \int_{-\infty\pm ia'}^{\infty\pm ia'} e^{\pm 2\pi i \ell x/h} w(x) \, dx \right|
  \le M e^{-2\pi \ell a/h}.
\end{equation}
In order to minimize $|I_{h,D}-I|$ in a certain sense, the $B_{k,D}$ will
be chosen to eliminate as many low-order exponential terms in the sums
over $\ell$ as possible.
To nullify both terms a particular $\ell$ value, we require
\begin{equation}
\sum_{k=0}^D ( i\ell)^k B_{k,D} = 0 \quad\quad \text{and} \quad\quad
\sum_{k=0}^D (-i\ell)^k B_{k,D} = 0 .
\end{equation}
Applying this condition for $1\le\ell\le D/2$ matches the number of
unknown $B_{k,D}$ to the number of equations.
These equations are seen to be the same as (\ref{eq:vander-restricted-strip})
and the coefficients $B_{k,D}$ are thus identical to
those found previously in Theorem~\ref{thm:periodic_strip}, with
$B_{2m-1,D}=0$ and $B_{2m,D}$ given by (\ref{eq:B2m_infinite})
for $1\le m\le D/2$.
The bound (\ref{eq:bound_int_we}) then implies
\begin{equation}
|I_{h,D}-I| \le 2M \left| \sum_{\ell=D/2+1}^\infty e^{-2\pi\ell a/h}
  F_{\ell,D} \right|,
\end{equation}
where $F_{\ell,D}$ is defined by (\ref{eq:F_ell_binomial}) and we have
used the fact that all $F_{\ell,D}$ have the same sign for $\ell>D/2$ to
justify placing the absolute value outside the summation.
This equation immediately leads to the bounds (\ref{eq:bound_infinite})
and (\ref{eq:bound_infinite_asymp}).

To show the sharpness of the constant $2M$ in the bound, it is helpful
to employ the Fourier transform of $w(x)$,
\begin{equation}
\hat{w}(\xi) = \frac{1}{2\pi}\int_{-\infty}^\infty
  e^{-i\xi x} w(x) \, dx .
\end{equation}
Applying the Fourier transform to $w_{h,D}(x)$ and using the
Poisson summation formula~\cite[6.10.IV]{Hen77}, one obtains
\begin{equation}
I_{h,D}-I = 2\pi\sum_{\ell=D/2+1}^\infty F_{\ell,D}
 \left[ \hat{w}(2\pi\ell/h) + \hat{w}(-2\pi\ell/h)\right] ,
\end{equation}
which is the analog of (\ref{eq:IND-I-Fell}). For the function
\begin{equation}
w(x)=\frac{\cos[2\pi(D/2+1)x/h]}{x^2+L^2}, \quad L>0,
\end{equation}
we have
\begin{equation}
\hat{w}(\xi) = \frac{1}{2L}\left\{\begin{array}{ll}
  \cosh[2\pi(D/2+1)L/h]e^{-|\xi|L} & |\xi| \ge 2\pi(D/2+1)/h \\[0.3em]
  e^{-2\pi(D/2+1)L/h} \cosh(\xi L) & \text{otherwise}
  \end{array} \right.
\end{equation}
and
\begin{align} \label{eq:IhD-exact}
I_{h,D}-I &= \frac{2\pi}{L} \cosh[2\pi(D/2+1)L/h]
  \sum_{\ell=D/2+1}^\infty F_{\ell,D} \, e^{-2\pi\ell L/h} \\ \nonumber
&\sim \frac{\pi}{L} (-1)^{D/2} \binom{D+1}{D/2}, \quad h \rightarrow 0.
\end{align}
For any $a$ with $0<a<L$,
\begin{equation}
\int_{-\infty}^\infty |w(x\pm ia)|\,dx \le \cosh[2\pi(D/2+1)a/h] \, J(a),
\end{equation}
where
\begin{equation}
J(a)=  \int_{-\infty}^\infty \frac{dx}{\sqrt{(x^2-a^2+L^2)^2 + 4a^2x^2}}
  = \frac{\pi}{L}\left[1 + \frac{a^2}{4L^2} + O(a^4) \right] .
\end{equation}
In the limit that $h,a\rightarrow 0$ with $h=o(a)$, the bounds
(\ref{eq:bound_infinite}) and (\ref{eq:bound_infinite_asymp})
are seen to be asymptotic to the exact result (\ref{eq:IhD-exact}).
\end{proof}

With the inclusion of derivative information the error of the trapezoidal
rule is seen to be improved from $|I_{h,D}-I|=O(e^{-2\pi a/h})$ to
$O(e^{-2\pi(D/2+1)a/h})$ as $h\rightarrow 0$.
The weights of the derivatives in the quadrature rule are the same as
those found in Theorem~\ref{thm:periodic_strip} for a periodic function
analytic within a strip and are given in Table \ref{tab:b_coeff}
for $D=2$, 4, and 6.

This quadrature rule appears to have first
been given by Kress in 1972~\cite{Kre72b}.
Our error bound is somewhat tighter, as Kress used an estimate analogous
to (\ref{eq:kress_approx}) in deriving his bound.
Other discussions of this quadrature rule are given in
Olivier and Rahman~\cite{Oli86} and Dryanov~\cite{Dry90,Dry92}.
The latter references also consider the case when derivatives are skipped
in the summation over $k$ in (\ref{eq:IhD}).
The error bound (\ref{eq:bound_infinite}) agrees with the result of
Dryanov~\cite[(3.11)]{Dry92}.

As alluded to in the above discussion of the sharpness of the error bound,
this quadrature rule may also be deduced using the Fourier transform
and Poisson summation formula.
This is also the approach taken in Ref.~\cite{Dry90}.
As noted by Trefethen and Weideman~\cite{Tre14}, this method seems to
require that a more stringent condition be placed on $w(x)$.

Bailey and Borwein~\cite{Bai06} have derived an error estimate for the
standard trapezoidal rule (\ref{eq:trap_infinite}) from the
Euler-Maclaurin formula. For an infinite integration interval, their
equation~(3) in our notation reads
\begin{equation}
\mathcal{E}_2(h,m) = h(-1)^{m+1}\left(\frac{h}{2\pi}\right)^{2m}
  \sum_{j=-\infty}^\infty
  w^{(2m)}(x_j)
\end{equation}
and the corresponding bound on the remaining error is given as
\begin{equation} \label{eq:bound_bailey}
\begin{split}
|I_h+\mathcal{E}_2&(h,m)-I| \le \\ & 2[\zeta(2m)+(-1)^{2m}\zeta(2m+2)]
  \left(\frac{h}{2\pi}\right)^{2m}
  \int_{-\infty}^\infty \left|w^{(2m)}(x)\right|\, dx ,
\end{split}
\end{equation}
where $\zeta$ is the Riemann zeta function.
They note that the estimate $\mathcal{E}_2(h,m)$ is ``very accurate.''
The quantity $\mathcal{E}_2(h,m)$ corresponds exactly in our formalism to the
derivative correction resulting from taking all
$B_k=0$ in (\ref{eq:IhD}), except $B_0=1$ and $B_{2m}=(-1)^{m+1}$,
which nullifies the leading order term in the error, resulting in
\begin{equation} \label{eq:bound_bailey_crb}
|I_h+\mathcal{E}_2(h,m)-I| \le 2M\left| \Li_0(e^{-2\pi a/h}) +
  (-1)^{m+1}\Li_{-2m}(e^{-2\pi a/h}) \right| ,
\end{equation}
and for $h\rightarrow 0$
\begin{equation} \label{eq:bound_bailey_crb_asymp}
|I_h+\mathcal{E}_2(h,m)-I| \le 2M (2^{2m}-1)e^{-4\pi a/h}
  \left[1+O(e^{-2\pi a/h})\right] .
\end{equation}
The formulations of the respective error bounds (\ref{eq:bound_bailey})
and (\ref{eq:bound_bailey_crb}) are observed to be quite different.
The bounds based on our derivative-free formalism
clearly show that including the derivative information leads to
an improvement in the geometric rate of convergence.
Bailey and Borwein noted that $\mathcal{E}_2(h,1)$ was always more accurate
than $\mathcal{E}_2(h,m)$ with $m>1$,
an observation that is likely explained by the factor of $(2^{2m}-1)$ in
the bound (\ref{eq:bound_bailey_crb_asymp}).

We conclude this section with the real-line analog of
Theorem~\ref{thm:periodic_restricted}, which is given without proof.
In practice, its applicability is limited and it is thus primarily
included for completeness.

\begin{theorem}
\label{thm:infinite_restricted}
Suppose $w$ is analytic in the half-plane $\Im\, x >-a$ for some $a>0$,
$w(x)\rightarrow 0$ uniformly as $|x|\rightarrow\infty$,
and for some $M$, it satisfies
\begin{equation}
\int_{-\infty}^\infty |w(x+ib)|dx \le M
\end{equation}
for all $b> -a$.
Further suppose that $D$ is a positive integer, $k$
is an integer with $0\le k\le D$, and
\begin{equation}
i^k A_{k,D} = \frac{(-1)^D}{D!} s(D+1,k+1) ,
\end{equation}
where $s(D+1,k+1)$ are the Stirling numbers of the first kind.
Then for $h>0$, $I_{h,D}$ as defined in (\ref{eq:IhD}) exists and
\begin{equation} \label{infite_half_plane}
|I_{N,D}-I| \le \frac{M}{(e^{2\pi a/h}-1)^{D+1}}
\end{equation}
and the constant $M$ is as small as possible.
\end{theorem}

\section{Large \texorpdfstring{$D$}{D} limit of the coefficients}
\label{sec:Dinf}

It was noted above in (\ref{eq:A_1}) that the $A_{k,D}$ coefficients
diverge as $D\rightarrow\infty$.
This is not the case for the coefficients $B_{2m,D}$.
Considering $F_{\ell,D}$ to be an analytic function of $\ell$,
(\ref{eq:F_ell}) becomes
\begin{equation}
F_{\ell,D\rightarrow\infty} = \prod_{m=1}^{\infty}\left[1-(\ell/m)^2\right] =
 \frac{\sin \ell\pi}{\ell\pi} =
  \sum_{m=0}^\infty (-1)^m \frac{(\ell\pi)^{2m}}{(2m+1)!}
\end{equation}
where Euler's product formula and the Taylor series for
$(\sin\ell\pi)/(\ell\pi)$ have both been utilized.
The coefficients can now be read off using (\ref{eq:F_ell}):
\begin{equation} \label{B_infinity}
B_{2m,D\rightarrow\infty} = \frac{\pi^{2m}}{(2m+1)!}.
\end{equation}
The coefficients $B_{2m,D}$ thus approach fixed values as
$D\rightarrow\infty$.
We note in passing that this result provides identities for the
infinite sums associated the $D\rightarrow\infty$ limits of
$B_{2m,D}$ with $m$ fixed. For example, (\ref{eq:sum_inverse_squared})
becomes the well-known sum $B_{2,D\rightarrow\infty} = \pi^2/6$.

The $D\rightarrow\infty$ limit can also be studied by considering
integrals on the real line\footnote{A similar analysis could also be done for
periodic functions analytic within a strip.}.
Assuming that $w(x)$ is analytic within the strip $|\Im\, x|\le h/2$
the integral (\ref{I_real_line}) may be written using Taylor series
expansions around the quadrature points as
\begin{align}
I &= \sum_{j=-\infty}^{\infty} \int_{x_j-h/2}^{x_j+h/2}
  \sum_{k=0}^\infty \frac{(x-x_j)^k}{k!}w^{(k)}(x_j)
  \, dx \\ \label{eq:IDinf}
  &= h \sum_{j=-\infty}^{\infty} \sum_{m=0}^\infty \left(\frac{h}{2}\right)^{2m}
       \frac{w^{(2m)}(x_j)}{(2m+1)!} .
\end{align}
By comparison with (\ref{eq:IhD}), we find $B_{2m,D\rightarrow\infty}$ as given by
(\ref{B_infinity}) and $B_{2m+1,D\rightarrow\infty}=0$.

Finally, we will consider the error terms
of Theorems~\ref{thm:periodic_strip} and~\ref{thm:infinite_strip}
in the $D\rightarrow\infty$ limit. Using Stirling's approximation
for the factorials,
\begin{equation} \label{eq:Dinf}
\binom{D+1}{D/2} \sim \frac{2^{D+2}}{\sqrt{2\pi D}}, \quad D\rightarrow\infty.
\end{equation}
For the case of Theorem~\ref{thm:periodic_strip},
the bound (\ref{eq:bound_periodic_strip_asymp}) indicates
\begin{equation}
\lim_{D\rightarrow\infty} |I_{N,D}-I| \rightarrow 0, \quad N>\frac{2\log2}{a}
\end{equation}
and for Theorem~\ref{thm:infinite_strip},
the bound (\ref{eq:bound_infinite_asymp}) shows
\begin{equation} \label{eq:lim_IhD}
\lim_{D\rightarrow\infty} |I_{h,D}-I| \rightarrow 0, \quad h<\frac{\pi a}{\log2} .
\end{equation}
In both cases, the convergence is geometric.
We also note that the requirement $h<\pi a/\log 2$ for (\ref{eq:lim_IhD})
is less restrictive than $h<2a$ which was assumed in the preceding paragraph.

\section{Other Approaches to Derivative Corrections}
\label{sec:other_approaches}

Here we discuss briefly two other approaches to derivative corrections to the
trapezoidal rule on the real line. They have a logical underpinning, but
are not optimal. Explicit error bounds will not be
derived, but it is clear that the improvement for these approaches scales
as a power of $h$, rather rather than exponentially.
In the appropriate limits, these methods will approach
Theorem~\ref{thm:infinite_strip}.
Here, we define the quadrature rule to be
\begin{equation} \label{eq:I_G}
I_G =  h \sum_{j=-\infty}^\infty \sum_{k=0}^D h^k \, G_{k} \, w^{(k)}(x_j) ,
\end{equation}
i.e., (\ref{eq:IhD}) but without the factors of $2\pi$.

One approach is to simply truncate the Taylor series expansion in
(\ref{eq:IDinf}), which results in
\begin{equation}
G_{k} = \left\{ \begin {array}{ll}
  [ 2^k(k+1)! ]^{-1} & \mbox{$k$ even} \\
  0 & \mbox{$k$ odd} \end{array} \right. .
\end{equation}
As noted above in section~\ref{sec:Dinf} , this rules does approach
Theorem~\ref{thm:infinite_strip} in the limit that $D\rightarrow\infty$,
i.e., when the full Taylor series is utilized.

Another approach is based upon interpolating polynomials.
We consider $2N$ points with equal spacing $h$.
The values of $w(x)$ and its first $D$ derivatives at the $2N$ points
can be described by a unique polynomial of degree $P=2N(D+1)-1$,
which is a particular implementation of the Hermite interpolating polynomial.
Rather than determine the polynomial coefficients, we will work directly
with the coefficients of the quadrature rule.
Assuming the points to be centered about $x=0$, a quadrature rule
for the integral between the two central points may be written as
\begin{equation} \label{eq:stencil}
\begin{split}
\int_{-h/2}^{h/2} w(x)\, dx = h\sum_{i=1}^N & \sum_{k=0}^D h^k \left[
  g^-_{ik} w^{(k)}(-\frac{2i-1}{2}h) + \right. \\
  & \left. g^+_{ik} w^{(k)}(\frac{2i-1}{2}h) \right] .\,
\end{split}
\end{equation}
where the $g^\pm_{ik}$ are unknown coefficients.
Since the monomials $x^p$ with $0\le p\le P$ form a linearly
independent and complete basis for all polynomials up to the degree
of the desired interpolating polynomial, the unknown coefficients may
be determined by requiring that that the quadrature rule evaluates
these monomials exactly~\cite{Bur12}:
\begin{equation}
\begin{split}
\int_{-h/2}^{h/2} x^p \, dx =  h\sum_{i=1}^N & \sum_{k=0}^{\min(D,p)} h^k \left[
  g^-_{ik} \frac{p!}{(p-k)!} \left(-\frac{2i-1}{2}h\right)^{p-k} + \right. \\
  & \left. g^+_{ik} \frac{p!}{(p-k)!} \left(\frac{2i-1}{2}h\right)^{p-k}
  \right] .
\end{split}
\end{equation}
Since the integral on the left-hand side vanishes when $p$ is odd, we have
\begin{equation} \label{eq:aik_minus}
g^-_{ik}=(-1)^k g^+_{ik} ,
\end{equation}
and for $p$ even
\begin{equation} \label{eq:g_plus}
\frac{1}{p+1} = 2 \sum_{i=1}^{N} \sum_{k=0}^{\min(D,p)} g^+_{ik}
  \frac{p!}{(p-k)!} (2i-1)^{p-k} 2^k .
\end{equation}
This linear system may be solved for $g^+_{ik}$.
A trapezoidal rule for the real line may then be derived by building
up a composite rule using (\ref{eq:stencil}) as stencil which is translated
as needed to integrate each subinterval.
This procedure results in
\begin{equation}
G_{k} = \sum_{i=1}^N g^-_{ik}+g^+_{ik},
\end{equation}
where $G_{k}$ is defined in (\ref{eq:I_G}) and is understood to
depend on $N$ and $D$.
For $k$ odd, $G_{k}$ vanishes because of (\ref{eq:aik_minus}).
For $k=0$, (\ref{eq:g_plus}) with $p=0$ gives $G_0=1$.
The results for $G_{k}$ for $D=2$ and~4 are shown for a range of
$N$ in Table~\ref{tab:poly_interp}.
It should be noted that the linear system (\ref{eq:g_plus}) is poorly
conditioned and must be solved carefully; we utilized exact rational
arithmetic for calculating $G_{k}$.

\begin{table}[tp]
{\footnotesize
  \caption{The coefficients $G_{k}$ for $D=2$, $D=4$, and selected $N$
    values. The last line gives the $B_{k,D}/(2\pi)^k$ values.
    The final digits are rounded. }
\label{tab:poly_interp}
\begin{center}
\begin{tabular}{c|c|cc}
    & $D=2$ & \multicolumn{2}{c}{$D=4$} \\
$N$ & $G_{2}$ & $G_{2}$ & $G_{4}$ 
\rule[-0.5em]{0pt}{0.5em} \\ \hline
\rule{0pt}{1.0em}%
  1 & 0.01666667 &   0.02777778 &   0.00006614 \\
  2 & 0.02239658 &   0.02980321 &   0.00011332 \\
  3 & 0.02426698 &   0.03068087 &   0.00013553 \\
  4 & 0.02493071 &   0.03112776 &   0.00014685 \\
  6 & 0.02527042 &   0.03149554 &   0.00015617 \\
  8 & 0.02532091 &   0.03160842 &   0.00015903 \\
 10 & 0.02532879 &   0.03164473 &   0.00015995 \\
 15 & 0.02533028 &   0.03166164 &   0.00016037 \\
 20 & 0.02533030 &   0.03166278 &   0.00016040 \\ \hline
$\rightarrow\infty$  &  0.02533030 & 0.03166287 & 0.00016041 \\ \hline
\end{tabular}
\end{center}
}
\end{table}

The last line of Table~\ref{tab:poly_interp} provides $B_{k,D}/(2\pi)^k$,
the optimal values from Theorem~\ref{thm:infinite_strip}.
It is seen that as $N$ increases, $G_{k}$ approaches these optimal values.
This result is not surprising, since the large-$N$ limit of polynomial
interpolation without derivatives is cardinal or sinc
interpolation~\cite{Whi15}, which with the inclusion of derivatives
generalizes to cardinal Hermite interpolation~\cite{Kre72b},
which in turn can be used to derive the optimal quadrature formulas
given here~\cite{Kre72b}.
Although we have not proven that the large-$N$ limit of $G_{k}$ is
$B_{k,D}/(2\pi)^k$, it is very likely to be the case and is observed
in practice.

\section{Conclusions}
\label{sec:conclusions}

Trapezoidal rules including derivative information have been derived
for periodic integrands or for integrals over the entire real line,
for functions which are analytic in a half plane or within a strip
including the path of integration.
The error bounds for the various cases, (\ref{eq:bound_periodic_plane}),
(\ref{eq:bound_period_strip_both}), (\ref{eq:boud_infinite_both}), and
(\ref{infite_half_plane}), are seen all seen to have similar structure. 
The quadrature rules converge geometrically as both the number of
quadrature points and number of included derivatives are increased.
Generally speaking, the inclusion of additional quadrature points, or
additional derivatives, are equally valuable for improving accuracy.
These observations support the statement made in the introduction
that the inclusion of derivative information in the quadrature rule
is is most likely to be useful when the computational effort required
to obtain the derivatives is significantly less than for additional
quadrature points.
For the case of integrands analytic within a strip, the quantity
$F_{D/2+1,D}$, which governs the leading behavior of the error, does
according to (\ref{eq:Dinf}) also grows geometrically with $D$ as
$D\rightarrow\infty$, which implies there is a significant penalty
for utilizing large $D$ values.
We also note that the analytic strip cases are more likely to be useful
in practice, as they are applicable to a much broader class of functions.

\section*{Acknowledgments}

We thank Rainer Kress for a useful discussion regarding
his work on this topic~\cite{Kre72a,Kre72b}.

\bibliographystyle{siamplain}
\bibliography{num_quad}

\begin{thebibliography}{10}

\bibitem{Bai06}
{\sc D.~H. Bailey and J.~M. Borwein}, {\em Effective bounds in
  {E}uler-{M}aclaurin-based quadrature ({S}ummary for {HPCS06})}, in 20th
  International Symposium on High-Performance Computing in an Advanced
  Collaborative Environment (HPCS'06), May 2006, pp.~34--34,
  \url{https://doi.org/10.1109/HPCS.2006.22}.
\newblock An expanded version of this paper is available at
  \url{https://escholarship.org/uc/item/0sx8r4sq}.

\bibitem{Bre10}
{\sc D.~M. Bressoud}, {\em Combinatorial analysis}, in NIST Handbook of
  Mathematical Functions, F.~W. Olver, D.~W. Lozier, R.~F. Boisvert, and C.~W.
  Clark, eds., Cambridge University Press, New York, NY, USA, 1st~ed., 2010.

\bibitem{Bur12}
{\sc C.~O.~E. Burg}, {\em Derivative-based closed {N}ewton-{C}otes numerical
  quadrature}, Applied Mathematics and Computation, 218 (2012), pp.~7052 --
  7065, \url{https://doi.org/10.1016/j.amc.2011.12.060}.

\bibitem{Dav84}
{\sc P.~J. Davis and P.~Rabinowitz}, {\em Methods of Numberical Integration},
  Academic Press, Orlando, FL, 2nd~ed., 1984.

\bibitem{Dry90}
{\sc D.~P. Dryanov}, {\em Quadrature formulae for entire functions of
  exponential type}, Journal of Mathematical Analysis and Applications, 152
  (1990), pp.~488--495, \url{https://doi.org/10.1016/0022-247X(90)90079-U}.

\bibitem{Dry92}
{\sc D.~P. Dryanov}, {\em Optimal quadrature formulae on the real line},
  Journal of Mathematical Analysis and Applications, 165 (1992), pp.~556--564,
  \url{https://doi.org/10.1016/0022-247X(92)90059-M}.

\bibitem{Hen77}
{\sc P.~Henrici}, {\em Applied and Computational Complex Analysis, Vol. 2:
  Special Functions, Integral Transforms, Asymptotics, Continued Fractions},
  Wiley, New York, 1977.

\bibitem{Kre72a}
{\sc R.~Kre{\ss}}, {\em On general {H}ermite trigonometric interpolation},
  Numerische Mathematik, 20 (1972), pp.~125--138,
  \url{https://doi.org/10.1007/BF01404402}.

\bibitem{Kre72b}
{\sc R.~Kress}, {\em On the general {H}ermite cardinal interpolation},
  Mathematics of Computation, 26 (1972), pp.~925--933,
  \url{https://doi.org/10.1090/S0025-5718-1972-0320586-6}.

\bibitem{Oli86}
{\sc P.~Olivier and Q.~I. Rahman}, {\em Sur une formule de quadrature pour des
  fonctions enti\`eres}, ESAIM: M2AN, 20 (1986), pp.~517--537,
  \url{https://doi.org/10.1051/m2an/1986200305171}.

\bibitem{Tre14}
{\sc L.~N. Trefethen and J.~A.~C. Weideman}, {\em The exponentially convergent
  trapezoidal rule}, SIAM Review, 56 (2014), pp.~385--458,
  \url{https://doi.org/10.1137/130932132}.
\newblock Comments and errata:
  \url{http://appliedmaths.sun.ac.za/~weideman/SIREVerrata.pdf}.

\bibitem{Whi15}
{\sc E.~T. Whittaker}, {\em {XVIII}. {O}n the functions which are represented
  by the expansions of the interpolation-theory}, Proceedings of the Royal
  Society of Edinburgh, 35 (1915), pp.~181--194,
  \url{https://doi.org/10.1017/S0370164600017806}.

\bibitem{Wil78}
{\sc D.~R. Wilhelmsen}, {\em Optimal quadrature for periodic analytic
  functions}, SIAM Journal on Numerical Analysis, 15 (1978), pp.~291--296,
  \url{https://doi.org/10.1137/0715020}.

\end{thebibliography}

\end{document}